\documentclass[11pt]{amsart} 
\usepackage{amsthm,amsbsy,amsmath,amssymb,amscd,amsfonts,array}
\usepackage{mathrsfs,verbatim,enumerate,enumitem,xypic,color,graphicx}
\usepackage[all,cmtip]{xy}
\xyoption{arc}
\usepackage{hyperref}
\usepackage{arydshln}

\newtheorem{thm}{Theorem}[section]
\newtheorem{prop}[thm]{Proposition}
\newtheorem{lem}[thm]{Lemma}
\newtheorem{cor}[thm]{Corollary}

\theoremstyle{definition}

\newtheorem{question}[thm]{Question}

\newtheorem{rem}[thm]{Remark}

\numberwithin{equation}{section}

\DeclareFontFamily{OT1}{rsfs}{}
\DeclareFontShape{OT1}{rsfs}{n}{it}{<-> rsfs10}{}
\DeclareMathAlphabet{\curly}{OT1}{rsfs}{n}{it}

\newcommand{\NN}{\mathbb{N}} 
\newcommand{\ZZ}{\mathbb{Z}} 
\newcommand{\QQ}{\mathbb{Q}} 
\newcommand{\rpk}{{\mathbb {R}}{{\mathbb {P}}^2}}
\newcommand{\cpk}{{\mathbb {C}}{{\mathbb {P}}^2}}
\newcommand{\cpegy}{{\mathbb {C}}{{\mathbb {P}}^1}}
\newcommand{\cpkk}{{\overline {{\mathbb C}{\mathbb P}^2}}}

\newcommand{\be}{\begin{eqnarray*}}
\newcommand{\ee}{\end{eqnarray*}}
\newcommand{\bne}[1]{\begin{eqnarray} \label{#1} }
\newcommand{\ene}{\end{eqnarray}}

\newcommand{\bp}{
   \arraycolsep=6pt
   \def\arraystretch{1}
   \begin{pmatrix}
}
\newcommand{\ep}{
   \end{pmatrix}
   \arraycolsep=2pt
   \def\arraystretch{1.2}
}

\def\arraystretch{1.2}

\begin{document}

\author[G. Mati\'c]{Gordana Mati\'c}
\address{Department of Mathematics, University of Georgia, Athens}
\email{gordana@math.uga.edu}

\author[F. \"Ozt\"urk]{Fer\.{i}t \"Ozt\"urk}
\address{Department of Mathematics, Bo\u{g}azi\c{c}i University}
\email{ferit.ozturk@boun.edu.tr}

\author[J. Reyes]{Javier Reyes}
\address{Department of Mathematics, University of Maryland}
\email{jereyes@umd.edu}

\author[A. I. Stipsicz]{Andr\'as I. Stipsicz}
\address{HUN-REN R\'enyi Institute of Mathematics, Budapest}
\email{stipsicz.andras@renyi.hu}

\author[G. Urz\'ua]{Giancarlo Urz\'ua}
\address{Facultad de Matem\'aticas, Pontificia Universidad Cat\'olica de Chile}
\email{gianurzua@gmail.com}

\date{\today}
\title[An exotic $5\rpk$ in the 4-sphere]{An exotic $5\rpk$ in the 4-sphere}

\begin{abstract}
  We show an example of an embedded copy of $5\rpk$ in the four-sphere
  $S^4$ which is topologically standard but smoothly knotted,
  i.e. smoothly not isotopic
  to the standard embedding.
\end{abstract}

\maketitle

\section{Introduction}
\label{sec:intro}

The real projective plane $\rpk$ can be embedded in the four-sphere
$S^4$ by simply capping off a M\"obius band in $S^3$ having unknotted
boundary, with a slice disk of its boundary unknot. Indeed, in this
way we get two different embeddings, as there are two M\"obius bands
with the property above. The two $\rpk$'s will have normal Euler
numbers $-2$ and $2$. Another way to present these embeddings is to
take the quotient of the complex projective plane $\cpk$ (or $\cpkk$)
with the involution given by complex conjugation in homogeneous
coordinates~\cite{Massey}.  The quotient is
$S^4$, and the fixed point set is the real projective
plane $\rpk\subset \cpk$. Connect summing the two embeddings above in
$S^4$ we get embeddings of $n\rpk$ into $S^4$ with Euler numbers $-2n,
-2n+4, \ldots , 2n-4, 2n$. These embeddings are the \emph{standard
  embeddings} with the appropriate Euler numbers. We note that the
complement of the standardly embedded $n\rpk$ has fundamental
group isomorphic to $\ZZ/2\ZZ =\ZZ _2$.

In \cite{FKV1, FKV2} S. Finashin and M. Kreck showed that the 10-fold connected sum
$10\rpk$ admits (indeed, infinitely many) embeddings with Euler number
$\pm 16$ into $S^4$ which are topologically isotopic,
but smoothly not.  The argument utilized the discovery of exotic
smooth structures on the four-manifold $\cpk \# 9\cpkk$ by
S.~Donaldson~\cite{Don}. This construction was improved by S.~Finashin
in \cite{F} to exotic embeddings of $6\rpk$ with Euler number $\pm 8$
--- the basis of the argument was similar to the one above: in this
case, the existence of exotic smooth structures on $\cpk \# 5\cpkk$
that were found in \cite{PSS}. Further exotic embeddings of
$n\rpk$ with $n=6,7,8,9$ and other Euler numbers were found by
A.~Havens~\cite{Havens}.

The common theme of the proofs of the results mentioned above was that
the exotic smooth manifolds can be presented as double branched covers
of $S^4$ branched along a connected sum of some number of copies of
$\rpk$; the fact that the branched covers are not diffeomorphic then
shows the non-existence of smooth isotopies (or even diffeomorphisms
of pairs) between the branching loci.  This method constructed
infinitely many pairwise smoothly non-isotopic embeddings. An
adaptation of the method of M.~Kreck~\cite{Kreck} then verified that
there are at most finitely many topological (locally flat) isotopy
classes, providing infinitely many topologically isotopic, smoothly
distinct smooth embeddings.

In this paper we show
\begin{thm}\label{thm:main}
  There is an embedding $\iota \colon 5\rpk \hookrightarrow S^4$ with
  Euler number $6$ which is
  smoothly non-isotopic to the standard embedding, such that
  the complement $S^4\setminus \iota (5\rpk)$ has fundamental group
  isomorphic to $\ZZ _2=\ZZ/2\ZZ$.
\end{thm}

{We combine this theorem with a recent result of
A.~Conway, P.~Orson and M.~Powell~\cite[Theorem~A]{COP},
stating that a locally flat embedding of $n\rpk$
into $S^4$ with the fundamental group of the complement
being $\ZZ _2$ and the Euler number being less than $2n$ in absolute
value is topologically (locally flat) isotopic to the standard embedding.
As a result, we get}

\begin{cor}
  There is an embedding of $5\rpk$ into $S^4$ with
  normal Euler number $6$ which is topologically
  isotopic to the standard embedding, but smoothly non-isotopic to it.
  \end{cor}
\begin{proof}
  According to \cite[Theorem~A]{COP} an embedding of
  $5\rpk$ with Euler number $6$ and complement with fundamental
  group $\ZZ _2$ is topologically isotopic to the standard embedding
  with the same Euler number.
  The corollary then follows from Theorem~\ref{thm:main}.
  \end{proof}

\begin{rem}
  Theorem~A of \cite{COP} indeed implies that all  the embeddings found in
  \cite{F, FKV1, FKV2} are exotic, that is, topologically isotopic
  to the standard embedding,
  but smoothly not isotopic to it.
  \end{rem}

Recently a very interesting preprint \cite{otherpaper} of J.~Miyazawa
appeared on the arXiv, claiming a much stronger result: the existence
of infinitely many exotic embeddings of a single $\rpk$ in $S^4$. The
techniques of the proof, and indeed the embeddings found in
\cite{otherpaper} are, however, very different from the one we discuss
in this paper: Miyazawa's exotic examples are connected sums of the
standard $\rpk$ with knotted spheres, hence the double branched covers
all contain spheres generating the second homology, implying that the
Seiberg-Witten invariants of the branched covers are trivial. In fact,
the proof of the fact that the embeddings are smoothly distinct relies
on a novel invariant (called Real Seiberg-Witten invariant in
\cite{otherpaper}). Indeed, it was shown that the double branched
covers along these smoothly knotted $\rpk$'s are all standard
(diffeomorphic to $\cpk$ with one of its orientations) \cite{maggie},
and the involutions generated by the branched covers are exotic. The
following question is still open:

  \begin{question}
    Is there a smooth exotic $\rpk\subset S^4$ which is not the connected
    sum of the standard $\rpk$ with a knotted sphere?
    \end{question}

The main ingredient of the proof of Theorem~\ref{thm:main}
is the construction of an
exotic smooth structure on
$\cpk \# 4\cpkk$  via the rational blow-down method
in the fashion discussed in \cite{R, RU}.
By choosing the appropriate construction, the implementation of
Finashin's methods from \cite{F} leads to the proof of
Theorem~\ref{thm:main}.  The crucial observation in this construction
is that in the pencil defining the elliptic fibration on $\cpk \# 9 \cpkk$
the generating curves and all the points to be blown up in the further
blow-ups are
real, so that one can perform each blow-up equivariantly.  Moreover,
as in \cite{F}, one can perform the rational blow-downs equivariantly,
thus the resulting four-manifold has a natural involution, with the
fixed point set $\widehat{F}\cong\#5\rpk$ and quotient $S^4$.
\begin{rem}
  Notice that Theorem~\ref{thm:main} provides a single exotic
  embedding of $5\rpk$, while \cite{F} gives infinitely many distinct
  embeddings for $6\rpk$. This difference is due to the fact that at
  the moment we have only one example of an exotic $\cpk\# 4\cpkk$
  with an appropriate involution, while there are infinitely many
  exotic $\cpk\# 5\cpkk$ with this extra structure, constructed by
  knot surgeries which we were unable to perform in an equivariant way
  on our example.
\end{rem}

The paper is organized as follows. We describe the construction of the
exotic smooth structure on $\cpk\# 4\cpkk$ and verify the existence of
an embedding of $5\rpk\subset S^4$ with Euler number equal to 6 in
Section~\ref{sec:construction}.  The proof of
  Theorem~\ref{thm:main} is then concluded by calculating the
  fundamental group of the complement in Section~\ref{sec:fundgroup} and
  proving exoticness of the smooth structure on
  $\cpk \# 4\cpkk$ in the subsequent sections.  In
  Section~\ref{sec:SW} we give a proof using the Seiberg-Witten
  invariants, while Section~\ref{sec:AG} contains a proof with a more
  algebro-geometric flavor. 

{\bf Acknowledgements:} The core of this work was carried out while GM
and F\"O visited the Erd\H{o}s Center, during the special semester
{\em Singularities and Low Dimensional Topology} in Spring 2023.  GM
and F\"O would like to thank the Erd\H{o}s Center for the great
working environment and hospitality.  GM was partially supported by  NSF award DMS1612036. AS was partially supported by
NKFIH grant \emph{\'Elvonal} KKP144148 and by ERC Advanced Grant
KnotSurf4d.  GU was supported by the
FONDECYT regular grant 1230065 and a Marie S. Curie FCFP fellowship.

\section{Construction of the exotic four-manifold}
\label{sec:construction}

{ As we said in the introduction, we will construct our example via
  first constructing an exotic $\cpk \# 4\cpkk$ by performing a
  rational blow-down of an appropriate blow-up of $\cpk$, then acting
  by an involution which fixes an embedded $ \# 5\rpk$.  Similar
  further examples of exotic structures on $\cpk \# 4\cpkk$ are
  presented in \cite{R, RU}, and all of these are obtained by
  rationally blowing-down Wahl chains in appropriate blowups of $\cpk$
  that are found as results of a computer
  search. \footnote{\url{https://github.com/jereyes4/Wahl_Chains/}}}

The example we consider here, different from any of those in  \cite{R, RU},
has the added property which we need for our construction: that all
operations are invariant under conjugation.
This example 
might, of course, not be the only such that can be found using a similar computer search.

In $\cpk$ consider the lines
\begin{itemize}
\item $X=\{x=0\}$, $Y=\{y=0\}$, $Z=\{z=0\}$,
\item $A=\{y+z=0\}$, $B=\{x+z=0\}$, $C=\{x+y=0\}$,
\item $H=\{x-y=0\}$, $L=\{x+y+z=0\}$.
\end{itemize}

The pair of singular cubics $F_{\infty}=X\cup Y\cup Z$
and $F_{9}=A\cup B\cup C$ determine the pencil $\{ F_t\}_{t\in \cpegy}$
of cubic curves given by the cubic equations
\[
(x+y)(x+z)(y+z) + (t-9)xyz =0.
\]

\begin{figure}[htbp]
\begin{center}
\includegraphics[width=8.5cm]{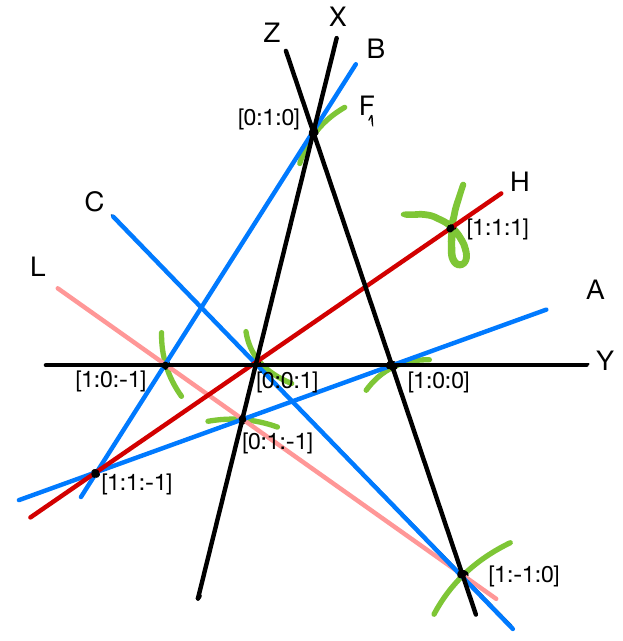}
\caption{$F_{\infty}$ is the black, $F_{9}$ is the blue curve. Part of
  $F_1$ is also shown (in green). The red line $H$ provides a
  multisection of the elliptic fibration we get after blowing up the
  base points of the pencil.}
\label{konf}
\end{center}
\end{figure}

This pencil has two further singular curves (besides $F_9$ and
$F_{\infty}$): $F_{1}$ with a nodal cubic singularity at $[1 : 1 : 1]$ and
$F_{10}$ which is the union of the line $L$ and the conic $$Q=\{xy+xz+yz\}.$$
The pencil has $6$ base points
$[0 : 0 : 1],\ [0 : 1 : 0],\  [1 : 0 : 0],\  [0 : 1 : -1],
\ [1 : 0 : -1],\ [1 : -1 : 0]$ with multiplicities $2,\ 2,\ 2,\ 1,\ 1$ and 1,
respectively.  Note that the curves generating the pencil are given by
polynomials with real coefficients, and all base points as
well as the singular point of $F_{1}$ are real.  As a
consequence, complex conjugation maps  curves in the pencil to
(potentially other) curves in the pencil.
As we will blow up only fixed points of the involution,
this involution will extend to the blown-up surface in a way that the new
exceptional divisors, as well as the strict transforms of invariant
curves will be also invariant under the involution.  Such blow-ups
will be referred to as \emph{equivariant}.

We blow up at the 6 base points 9 times to obtain an 
elliptic fibration $S \to \mathbb{C} \mathbb{P}^1$ with singular
fibers $I_6+I_3+I_2+I_1$ (see also \cite{HKK} or \cite{SSS} for
notation and conventions regarding the singular fibers in an elliptic
fibration).  Using the notation of \cite{R}, the 9 blow-ups, with the
exceptional curves (denoted by $E_i$, $i=1, \ldots , 9$), are
performed at the following intersection points:
\begin{eqnarray*}
& E_1:X\cap Y\cap C \cap H \cap F_1, \qquad E_2: E_1\cap C \cap F_1, \\
& E_3: X\cap Z\cap B \cap F_1, \qquad E_4: E_3 \cap B \cap F_1, \\
& E_5: Y\cap Z \cap A \cap F_1, \qquad E_6: E_5\cap A \cap F_1, \\
& E_7: X\cap A\cap L \cap F_1, \qquad E_8: Y\cap B \cap L \cap F_1, \qquad E_9: Z \cap C \cap L \cap F_1.
\end{eqnarray*}

Here and below (with a slight abuse of notation) we use the same letters for the
strict transforms of the curves after blow-ups. {We blow up equivariantly
13 more times as follows:
\begin{eqnarray*}
& E_{10}: E_1 \cap Y, \qquad
E_{11}: E_3 \cap Z, \qquad  E_{12}:E_4 \cap B,  \qquad E_{13}: E_4 \cap F_1, \\
& E_{14}: E_{13} \cap F_1 \qquad E_{15}: E_5 \cap E_6, \qquad E_{16}: E_9 \cap Z, \qquad E_{17}: E_9 \cap C \\
& E_{18}: A \cap B \cap H, \qquad E_{19}: L \cap H, \qquad E_{20}: E_{19} \cap H, \qquad E_{21}: E_{20} \cap H \\
& E_{22}: F_1 \cap H.
\end{eqnarray*}

Let $W$ be the resulting complex surface, and let $\phi \colon W \to S$ be the corresponding composition of 13 blow-ups.


\begin{figure}[h!]
\begin{center}
\includegraphics[width=13.3cm]{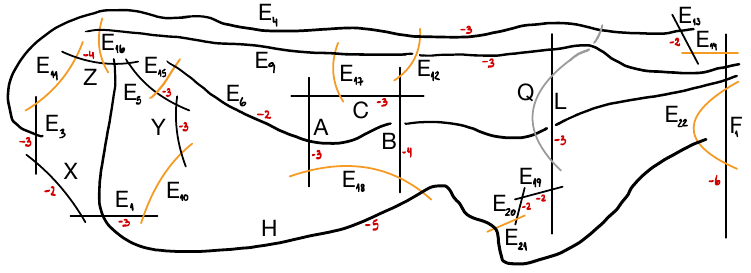}
\caption{The 22-fold equivariant blow-up $W$ of $\cpk$. The four singular fibers of the elliptic fibration $W \to \mathbb{C} \mathbb{P}^1$ are shown, together with the sections $E_4$, $E_9$, $E_6$, and the double-section $H$. The orange curves are $(-1)$-curves. The self-intersections of the curves in the chains $\mathcal{C}_1$ and $\mathcal{C}_2$
are in red.}
\label{22up}
\end{center}
\end{figure}

 \pagebreak

Thus we obtain the following pair of chains
$\mathcal{C}_1,\mathcal{C}_2$ of equivariant curves in $\cpk \# 22
\cpkk$. We include the self-intersection numbers of the spheres.


\[\begin{array}{*{19}c}
    \mathcal{C}_1 & \colon & B & - & C & - & A & - & E_6 & - & F_1 & - & E_9 & - & L & - & E_{19} & - & E_{20}\\
    \text{self-int}& \colon & -4 && -3 && -3 && -2 && -6 && -3 && -3 && -2 && -2\\
\end{array}\]
and
\[\begin{array}{*{19}c}
    \mathcal{C}_2 & \colon & Y & - & E_5 & - & Z & - & H & - & E_1 & - & X & - & E_3 & - & E_4 & - & E_{13}\\
    \text{self-int}& \colon & -3 && -3 && -4 && -5 && -3 && -2 && -3 && -3 && -2\\
\end{array}\]

Note that these chains are Wahl chains (i.e. are in
the family ${\mathcal {G}}$ of \cite{SSW})
and are  equivariant in the
complex surface $(W,c)$ diffeomorphic to $\cpk \# 22\cpkk$.
(Here $c$ is the involution given by the extension of complex conjugation
from $\cpk$ to the 22-fold equivariant blow-up.)
Moreover, with the $(-1)$-curves $E_{21}$ and $E_{14}$ we have longer,
equivariant chains $\mathcal{C}'_1=\mathcal{C}_1 \cup E_{21}$
and $\mathcal{C}'_2=\mathcal{C}_2 \cup E_{14}$.
The fixed point set $F$ of $c$ is diffeomorphic to
$\# 23\rpk$ embedded standardly in $W$,
with Euler number 42.
}

{ Recall that in the rational blow-down operation (introduced by
  Fintushel and Stern~\cite{FSRatBl} and extended to the generality
  used here by J.~Park~\cite{ParkRatBl}) one  considers a collection of
  embedded spheres in a four-manifold intersecting each other
  according to a linear plumbing tree, with self-intersections given
  by the continued fraction coefficients of $-\frac{p^2}{pq-1}$ for
  some relatively prime $p>q>0$.
  The boundary of a tubular
  neighborhood $C_{p,q}$ of this sphere configuration is the lens
  space $L(p^2, pq-1)$, and this lens space bounds a four-manifold
  $B_{p,q}$ with the rational homology of the disk.

  The rational blow-down operation replaces $C_{p,q}$ with $B_{p,q}$.  This
  four-manifold $B_{p,q}$ can be explicitly constructed by the
  following observation: the lens space $L(p^2, pq-1)$ is the double
  branched cover of $S^3$, branched along the two-bridge knot
  $K_{p^2,pq-1}$. The plumbing of annuli and M\"obius bands (with the
  right amount of twisting, dictated by the plumbing tree for
  $C_{p,q}$), provides an embedded surface $\Sigma_{p^2, pq-1}\subset
  S^3$ with boundary the two-bridge knot $K_{p^2,pq-1}$, and by
  pushing the interior of $\Sigma _{p^2, pq-1}$ into the disk $D^4$,
  we get $C_{p,q}$ as the double branched cover of $D^4$, branched
  along this pushed in $\Sigma _{p^2, pq-1}$. As it was shown in
  \cite{Lisca}, the two-bridge knots $K_{p^2,pq-1}$ are all slice (in
  fact, ribbon), and the double branched cover along a particular
  ribbon disk of this knot in $D^4$ provides $B_{p,q}$.}

    For our two chains above the corresponding  neighborhoods are
   $N( \mathcal{C}_1 )=C_{65,18}$ and    $N( \mathcal{C}_2 )=C_{79,30}$ with boundaries
   $L(65^2, 1169)$ and $L(79^2, 2369)$. The significance of these chains is in the following

\begin{thm}\label{thm:nondiffeo}
The symplectic 4-manifold $\widehat{W}$ obtained by rationally blowing
down the two disjoint Wahl chains $\mathcal{C}_1$ and $\mathcal{C}_2$
is an exotic $\cpk \# 4\cpkk$, i.e. it is homeomorphic to $\cpk \#
4\cpkk$ but not diffeomorphic to it.
\end{thm}
(In the rest of the paper a hat will always denote the corresponding
object after the blow-down.)

In our example rationally blowing down the two negative definite
chains $\mathcal{C}_1 , \mathcal{C}_2$ starting with the 22-fold
equivariant blow-up $W$ of $\cpk$, a manifold diffeomorphic to $\cpk
\# 22 \cpkk$, produces a manifold $\widehat{W}=W_0 \cup (B_{65,18}\cup
B_{79,30}))$ with $b_2^+=1$ and $b_2^-=4$, were $W_0= W \setminus
(\mathcal{C}_1 \cup \mathcal{C}_2))$.  Using the Seifert-Van Kampen
theorem it is not hard to show that

\begin{prop}
  The four-manifold
$\widehat{W}$ is simply
  connected.
\end{prop}
\begin{proof}
  To prove the statement of the proposition, it is enough to show that a loop
$\alpha$ around $E_{13}$, and a loop $\beta$ around $E_{20}$ are
homotopically trivial in $W_0$. This is because these loops generate
the fundamental groups of the links (i.e. the boundaries of the
tubular neighbourhoods of the chains) corresponding to $\mathcal{C}_1$
and $\mathcal{C}_2$ respectively. We can relate via
conjugation the loop $\alpha$ with a loop around $F_1$ using the
2-sphere $E_{14}$, and the loop $\beta$ with a loop around $H$ using
the 2-sphere $E_{21}$. (At this point we use Mumford's computation
\cite{mum} to relate loops around exceptional curves with powers of a
generating loop of the fundamental group of the link.) As the orders
of the fundamental groups of these links are $65^2$ and $79^2$, they
are coprime, and so $\alpha$ and $\beta$ are trivial in $W_0$.
\end{proof}

By Freedman's celebrated classification
result \cite{Fr}, simple connectivity and a straighforward
Euler characteristic and signature calculation shows that

\begin{cor}
$\widehat{W}$ is homeomorphic to $\cpk \# 4\cpkk$.
  \end{cor}

We will show that $\widehat{W}$ is exotic in Sections~\ref{sec:SW} and
\ref{sec:AG}.

The crucial fact in our construction is that the rational blow-downs
can be performed in a way that the new manifold $\widehat{W}$ has a
natural involution $\widehat{c}$, with its fixed point set
$\widehat{F}$ diffeomorphic to $\#5\rpk$.
To see this involution, recall that, as branched covers, both
$C_{p,q}$ and $B_{p,q}$ come with an involution; indeed, these
involutions extend the complex conjugation on the complement of
$C_{p,q}$, cf.~\cite[Section~3.1]{F}.

{The quotient of $\cpk \# 22 \cpkk$ by the conjugation action is
  $S^4$, with the fixed point set $\# 23 \rpk$ descending to the
  standardly embedded $\# 23 \rpk$ with normal Euler class 42. The
  neigborhoods $N( \mathcal{C}_1 )=C_{65,18}$ and $N( \mathcal{C}_2
  )=C_{79,30}$ descend under this action to two standard $B^4$. The
  blow-down operation replaces $B^4=N( \mathcal{C}_i)/c$ by
  $B^4=B_{p_i,q_i}/\hat{c}$, hence the quotient under the induced
  action $\hat{c}$ on $\widehat{W}$ is also $S^4$.}

\begin{cor}\label{cor:firstpart}
The embedding of $\iota\colon 5\rpk\to S^4$ found above has Euler
number $6$ and is not smoothly isotopic
to the standard embedding.
\end{cor}

\begin{proof}
As $\widehat{W}\to S^4$ is a double branched cover, simple Euler
characteristic calculation determines the topology of the branch
locus. Its normal Euler number follows from the formula derived from
the $G$-signature theorem. If the branch set is isotopic to the
standard embedding, then we would find a diffeomorphism between
$\widehat{W}$ and $\cpk \# 4\cpkk$, contradicting
Theorem~\ref{thm:nondiffeo}.
\end{proof}

\begin{rem}\label{rem:symplectic} It was shown in \cite{PS, Sym} that if the surfaces in $C_{p,q}$ are symplectic
in a symplectic four-manifold
(and intersect each other orthogonally with respect to the symplectic
structure) then the result of the rational blow-down admits a symplectic
structure which coincides with the original one outside of a neighborhood
of $C_{p,q}$.
\end{rem}

\section{Fundamental group calculation}
    \label{sec:fundgroup}

    In order to complete the proof of Theorem~\ref{thm:main}, we need
    to determine the fundamental group of the complement of
    $\widehat{F}=\iota (5\rpk)\subset S^4$. This computation will rely
    on the Seifert-Van Kampen theorem. Recall, that if $X=X_1\cup X_2$
    and $X_1\cap X_2$ is path connected, then (after choosing the base
    point in $X_1\cap X_2$) the fundamental group $\pi _1 (X)$ can be
    presented as the amalgamated product $\pi _1(X_1)*_{\pi _1(X_1\cap
      X_2)}\pi _1(X_2)$, with maps $i_j\colon \pi _1 (X_1\cap X_2)\to
    \pi _1(X_j)$ ($j=1,2$) induced by embeddings.  In our argument we
    will use the following principle:

\begin{lem}\label{lem:FungGroupCalc}
      Suppose that for the amalgamated product $A*_QB$ of the groups $A,B$ over
      $Q$ with homomorphisms $\varphi _A\colon Q\to A$ and
      $\varphi _B\colon Q\to B$ we have
      \begin{itemize}
      \item $\varphi _B$ is onto, and
      \item $\ker \varphi _B\leq \ker \varphi _A$.
      \end{itemize}
      Then $A*_QB$ is isomorphic to $A$.
    \end{lem}

\begin{proof} Since $\ker\varphi_B
    \leqslant \ker\varphi_A$ and $\varphi_B$ is surjective, there is a
    map $r \colon B \to A$ such that $\varphi_A = r \circ
    \varphi_B$. Together with the identity $A \to A$, they induce a
    map $F \colon A \ast_Q B \to A$ which, when precomposed with the
    natural map $G_A \colon A \to A \ast_Q B$, gives the identity on
    $A$. Therefore $G_A$ is injective, hence to show that $G_A$ is an
    isomorphism, we need to show that it is surjective. Let $G_B
    \colon B \to A \ast_Q B$ be the other natural map, and let $b \in
    B$. Choose $q \in Q$ with $b = \varphi_B(q)$ and let $a =
    \varphi_A(q)$. By construction, $G_A(a) = G_B(b)$, so we deduce
    that all elements in the image of $G_B$ are in the image of
    $G_A$. Since $A \ast_Q B$ is generated by such elements in the
    image of either $G_A$ or $G_B$, it follows that $G_A$ is
    surjective.
\end{proof}

      We will adapt the principle above to our current setting as
      follows.  As discussed in \cite[Section~3]{F}, the quotient
      spaces $W/c$ and $\widehat{W}/\widehat{c}$ are diffeomorphic to
      $S^4$.  Moreover, as we have pointed out above, since
      $F=\mathrm{Fix}(c)$ is standardly embedded in $W$, we have that
      $\pi_1(W/c - F)= \pi _1 (S^4-F)=\ZZ_2$.  To conclude the proof
      of Theorem~\ref{thm:main}, we need to show that
      $\pi_1(S^4-\widehat{F})=\pi_1(S^4-F)$.  Obviously the branching
      surfaces $F$ and $\widehat{F}$ in $S^4$ are the same outside the
      domain of the two rational blow-downs.

To set up notation, let ${U}_1,{U}_2\subset W$ be the equivariant
plumbing neighborhoods (to be blown down rationally) of the chains
$\mathcal{C}_1$ and $\mathcal{C}_2$ respectively, and
set $V_i=U_i/c$.  During rationally
blowing down $\mathcal{C}_i$, ${U}_i$ is replaced with an equivariant
rational homology ball $\widehat{U}_i$ and
$\widehat{V}_i=\widehat{U}_i/\widehat{c}$ will denote the
corresponding 4-ball in $S^4$ ($i=1,2$);
see Figure~\ref{fig:ratbl}. We
recall that $\partial V_i\cap F$ is a 2-bridge knot $K_i$ in $S^3\cong
\partial V_i$ and stays intact after the blow-down.  (These knots can be
determined through the continued fractions given by the framings of
the spheres in ${\mathcal {C}}_i$,
$K_1$ is the 2-bridge knot $K_{65^2, 65\cdot 18-1}$ and $K_2$ is the 2-bridge
knot $K_{79^2, 79\cdot 30-1}$.)
Moreover $\widehat{F}\cap \widehat{V}_i$ is a
slice disk for $K_i$, that is, a smooth 2-disk embedded in the 4-ball
with $\partial (\widehat{F}\cap \widehat{V}_i)=K_i$; indeed, these disks
are ribbon.
\begin{eqnarray*}
& \xymatrix{ W = \cpk \# 22\cpkk \ar[d]_{/c} & \ar@{_{(}->}[l] U_i \ar[d]_{/c} \ar[r]^{r.b.d.} & \widehat{U}_i \ar[d]_{/\widehat{c}}  \ar@{^{(}->}[r] &
\widehat{W} \ar[d]_{/\widehat{c}}\\
F\subset S^4 & \ar@{_{(}->}[l] V_i  \ar[r] & \widehat{V}_i    \ar@{^{(}->}[r]
&  S^4 \supset \widehat{F}_i
}
\end{eqnarray*}

\begin{figure}[h!]
\begin{center}
\includegraphics[width=10cm]{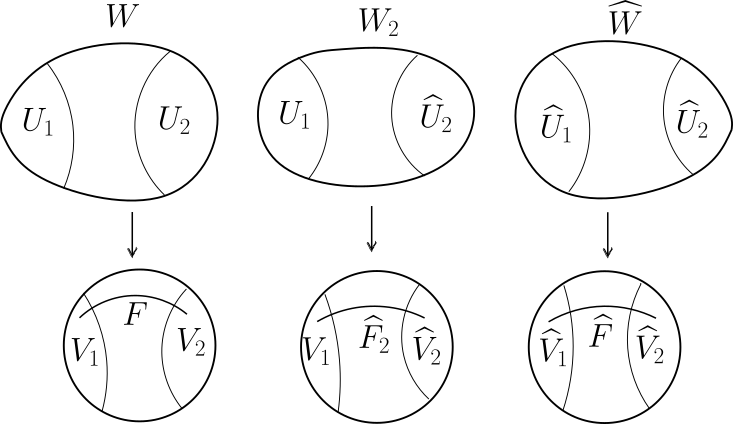}
\caption{Two rational blow-downs on $W$. The manifolds in the bottom row are
all $S^4$, containing the surfaces $F, \widehat{F}_2$ and $\widehat{F}$.}
\label{fig:ratbl}
\end{center}
\end{figure}

Consider the subsurfaces $F_i=F\cap V_i $ of the branch
locus $F$ in $S^4$.  $F_i$ is the plumbing (of the images under
quotient by $c$) of the neighborhoods of the invariant equators of the equivariant
 spheres that appear in the chain $\mathcal{C}_i$.
Since there exist spheres of odd intersection in each chain, $F_i$ is nonorientable
for $i=1,2$.


We will do the calculation in two steps; first we will blow down the
configuration ${\mathcal {C}}_2$ and then ${\mathcal {C}}_1$. In this
way we need to apply the Seifert-Van Kampen theorem several times.
Let us therefore define $W_2$ to be the four-manifold we get from $W$
by blowing down ${\mathcal {C}}_2$, by replacing the neighbourhood
$U_2$ in $W$ with ${\widehat{U}}_2$, and downstairs in $S^4$
replacing  $V_2$ with ${\widehat{V}}_2$. The branch set $F$ is also
changed,  as we replace $F_2$ with the slice disk of the
  slice knot $K_2$; we denote the new branch set by $\widehat{F}_2\subset
  S^4$.

\begin{lem}
\label{lemma:fundGp2ndRatBlDw}
  The fundamental group of the complement $S^4-\widehat{F}_2$ is isomorphic to
  $\ZZ _2$.
  \end{lem}
\begin{proof}
  The proof of this statement follows from the argument in
  \cite{F}. For the sake of completeness, we give a detailed account of the argument
  here.

  Let $I_2$ denote the fundamental group $\pi _1 ((S^4-V_2)-F)$,
  $J_2=\pi _1 (V_2-F)$, and ${\widehat{J}}_2=\pi _1
  ({\widehat{V}}_2-\widehat{F}_2)$.  With the notation $G_2=\pi
  _1(S^3-K_2)$ the application of the Seifert-Van Kampen theorem gives
  presentations of the fundamental groups before the blow-down (which
  is $\pi _1(S^4-F$)) and after the blow-down (which is $\pi _1
  (S^4-\widehat{F}_2)$) --- with the base point always chosen in the
  intersection $S^3-K_2$ --- as amalgameted products
  \[
  \pi _1 (S^4-F)\cong I_2*_{G_2}J_2, \qquad \pi _1 (S^4-\widehat{F}_2)\cong
  I_2*_{G_2}{\widehat{J}}_2.
  \]
  The maps $\varphi _{J_2}\colon G_2\to J_2$ and
  $\varphi _{{\widehat{J}}_2}\colon G_2\to {\widehat{J}}_2$ are
  both onto, as $F_2$
  is a pushed-in (nonorientable) spanning surface of $K_2$ and we replace
  it with a ribbon disk of $K_2$ (both the complement of a
  pushed-in surface and a ribbon disk can be
  built on the knot complement by adding only 2-, 3- and 4-handles).

  Note that $G_2$ is generated by two elements $\alpha $ and $\beta$
  (as $K_2$ is a two-bridge knot); we can take these generators as the
  meridians of $F_2$ at the two arcs $l_{\alpha}$ and $l_{\beta}$ of
  the intersection of $K_2$ with a neighborhood $\nu\gamma _0 \subset
  F$ where $\gamma_0$ is the boundary of the disk $E_{14}/c$,
  and is part of $F$.
  Once an
  orientation on $K_2$ is chosen, these normal circles come with an
  orientation.  As $F_2$ is a push-in of a spanning surface
  (so its complement can be built from 0-, 1- and 2-handles only), we have
  that $\alpha \beta ^{-1}$ is in $\ker \varphi _{J_2}$.
    Furthermore, the band $F_2\cap\nu\gamma _0$ gives $\alpha \beta \in
    \ker \varphi _{J_2}$.  As the quotient of $G_2$ with these
    relations is $\ZZ _2$, it follows that $\ker \varphi _{J_2}$
    (as a subgroup of $G_2$)
  is normally generated by these two elements.

  When $\varphi _{{\widehat{J}}_2}$ is composed with the
  abelianization map
\[
 G_2\stackrel{\varphi _{{\widehat{J}}_2}}{\longrightarrow}
 {\widehat{J}_2}\to H_1({\widehat{V}}_2-\widehat{F}_2; \ZZ
 )=H_1(S^3-K_2; \ZZ )=\ZZ,
\]
  we get that the kernel of $\varphi _{{\widehat{J}}_2}$ is contained
  in the normal subgroup generated by $\alpha \beta ^{-1}$.

  In order to apply Lemma~\ref{lem:FungGroupCalc} we need to show that
  $\ker \varphi _{I_2}$ contains both $\alpha\beta$ and $\alpha \beta
  ^{-1}$. Lemma~\ref{lem:FungGroupCalc} then shows that both groups
  $\pi _1 (S^4-F)$ and $ \pi _1 (S^4-\widehat{F}_2)$ are isomorphic to
  $I_2$, and since $\pi _1(S^4-F)\cong \ZZ _2$, this concludes the
  argument.

  { Recall that for ${\mathcal{C}}_1$ and ${\mathcal{C}}_2$ we have
    the extended chains ${\mathcal{C}}_1'$ and ${\mathcal{C}}_2'$ we
    get by adding the classes $E_{21}$ and $E_{14}$, respectively.
    Let $W(\mathcal{C})$ denote a characteristic set for a chain
    $\mathcal{C}$. For $i=1,2$, the sets $W(\mathcal {C}_i)$ and
    $W(\mathcal{C}'_i)$ are unique since the intersection matrices in
    all 4 cases have odd determinants.  The sets are shown below
    marked by a dot $\cdot$ at the top and by a cross $\times$ at the
    bottom for $\mathcal {C}_i$ and $\mathcal {C}'_i$ respectively:
  $$ (i=1): [\dot{4},3,\underset{\times}{3},\dot{2},\underset{\times}{6},\dot{3},3,\underset{\times}{2},2],\underset{\times}{1};
  \qquad (i=2): [\dot{3},\underset{\times}{3},4,\underset{\times}{5},\dot{3},2,\dot{3},\underset{\times}{3},2],\underset{\times}{1}. $$
}
  Let us denote by $S_i$ $(i=1,\ldots, 9)$ the
  $i$-th  sphere (from left to right) in the chain
  ${\mathcal {C}}_2$ and let $\gamma_i=\partial
  (S_i/c)\subset F$ be the image of the
  equatorial circle of $S_i$.  Recall that the 1-dimensional mod-2 cohomology
  class
\[
 PD( \sum_{S_i\in W(\mathcal{C}_2)}
 [\gamma_i])=PD([\gamma_1]+[\gamma_5]+[\gamma_7])
\]
  is the first Stiefel-Whitney class of $F_2$ and hence is the
  obstruction to orientability of $F_2$. (Here $PD(\alpha )$
  denoted the Poincar\'e dual of the homology class $\alpha$.)

  We observe that $W(\mathcal{C}_2)$ does not extend over the longer
  chain $E_{14}\cup\mathcal{C}_2$, or equivalently the obstruction to
  orientability of $F_2$ is not equal to that of the extended surface
  $F'_2=F_2\cup\nu\gamma _0$.  (The obstruction for $F'_2$ is
  $PD([\gamma_0]+[\gamma_2]+[\gamma_4]+[\gamma_8])$.)
  {This implies that $F'_2-\cup_{i=1}^{9}\nu\gamma_i$ is
    nonorientable and hence the surface $\widehat{F}'_2$ is a M\"obius
    band, which we denote by $M$.}  Consider the bands $\sigma_1=M -
  \widehat{V}_2$ and $\sigma_2=M\cap \widehat{V}_2$ which are glued
  together along the intervals $l_{\alpha}$ and $l_{\beta}$.  Choosing
  the base point in the intersection of $\partial V_2-K_2$ and the
  disk $E_{14}/c$, this disk provides the continuous family of arcs
  from the base point to the normal circles of the middle arc of these
  bands. This construction then provide the desired relations
  among the normal circles $\alpha$ and $\beta$, concluding the claim.
\end{proof}

Next we apply the other rational blow-down, blowing down
${\mathcal {C}}_1$ this time. The branch set $\widehat{F}_2$ is modified;
we denote the new branch set by $\widehat{F}$.

\begin{lem}
\label{lemma:fundGp1stRatBlDw}
  The fundamental group of the complement $S^4-\widehat{F}$ is
  isomorphic to $\ZZ _2$.
  \end{lem}
\begin{proof}
  We use the objects and notation of the previous proof; we add
  $I_1=\pi _1((S^4-V_1)-F)$ $J_1=\pi _1(V_1-F)$ and
  ${\widehat{J}}_1= \pi _1 ({\widehat{V}}_1-{\widehat{F}})$.  As
  before, we have $G_1=\pi _1(S^3-K_1)$, and the fundamental groups
  $\pi _1(S^4-F)$ and $\pi _1 (S^4-\widehat{F})$ can be expressed
  as amalgamated products over $G_1$ similar to the previous
  case. (Notice that now the base point is chosen in $S^3-K_1$.)  The
  maps $\varphi _{J_1}$ and $\varphi _{{\widehat{J}}_1}$ are onto as
  before, hence we only need to deal with the kernels --- and indeed,
  the kernels of these maps can be identified exactly in the same way
  as before.

  For the band $\sigma_1$ from the previous argument, the part
  $\sigma_1 - V_1$ has two connected components; say $\sigma_3$ and
  $\sigma_4$, with $l_{\alpha}=\sigma_3\cap K_2$ and
  $l_{\beta}=\sigma_4\cap K_2$. Let us also set
  $l'_{\alpha }=\sigma_3\cap K_1$ and $l'_{\beta}=\sigma_4\cap K_1$.
  The normal circles $\alpha '$ and $\beta '$ to $l'_{\alpha}$ and
  $l'_{\beta}$ generate $\pi _1(S^3-K_1)$.

    As in the previous proof, we have that the subgroup of $G_1$
    generated by $\alpha'\beta'$ and $\alpha'\beta'^{-1}$ contains
    both $\ker\varphi_{J_1}$ and $\ker\varphi_{\widehat{J}_1}$. So to
    apply Lemma~\ref{lem:FungGroupCalc} again and to conclude the
    proof of the lemma, it is enough to show that $\ker \varphi
    _{I_1}$ contains both $\alpha'\beta'$ and $\alpha' \beta'^{-1}$.

    In these arguments we proceed as before: we will find two bands
    connecting $l'_{\alpha}$ and $l'_{\beta}$, so that their union is
    a M\"obius band, and the normal circles to these bands will
    provide the required homotopies.  Consider first the union
    $\sigma_3\cup\sigma_2\cup\sigma_4$.  We can connect $l'_{\alpha}$
    and $l'_{\beta}$ in $(S^4-\widehat{V}_1)\cap \widehat{F}_2$ so
    that the neighbourhood of this arc will provide the other
    band. These two bands now show that both $\alpha '\beta '$ and
    $\alpha \beta '^{-1}$ are in $\ker \varphi _{I_1}$, concluding the
    proof.
\end{proof}

\begin{proof}[Proof of Theorem~\ref{thm:main}]
  Corollary~\ref{cor:firstpart} verifies the first claim in Theorem~\ref{thm:main}.
  Rationally blowing down the chain $\mathcal{C}_2$ and then the chain $\mathcal{C}_1$,
  Lemma~\ref{lemma:fundGp2ndRatBlDw} and
  Lemma~\ref{lemma:fundGp1stRatBlDw} prove that
  $\pi _1(S^4-F)\cong\pi _1(S^4-\widehat{F}_2)\cong\pi _1(S^4-\widehat{F})$.
  Since  $\pi _1(S^4-F)\cong \ZZ_2$, the claim of the theorem follows at once.
\end{proof}

\section{ Seiberg-Witten calculation}
\label{sec:SW}

{
  In this section, using Seiberg-Witten theory, we verify
  that the four-manifold $\widehat{W}$
constructed by rationally blowing down the two Wahl chains in $\cpk
\# 22\cpkk$ is, indeed, not diffeomorphic to
$\cpk \# 4\cpkk$.
}

\begin{thm}
  The four-manifolds
  $\widehat{W}$ and $\cpk \# 4\cpkk$
  are not diffeomorphic.
  \end{thm}
\begin{proof}
  We will use Seiberg-Witten invariants in the argument; below we will
  adapt the argument detailed in \cite[Section~3]{SSz}.

{
  Recall that for manifolds with $b_2^+>1$ the Seiberg-Witten
  invariants are diffeomorphism invariants, while for $b_2^+=1$ the
  invariants might depend on a chosen metric on the manifold and a
  perturbation term in the equations. By using small perturbation and
  considering manifolds with $b_2^-\leq 9$, however,
  this dependence disappears.  By
  the existence of a positive scalar curvature metric on $\cpk$ it
  easily follows that the (small perturbation) Seiberg-Witten
  invariants of $\cpk \# 4\cpkk$ are identically zero.

  Consider the homology class $K=3h-\sum $~(all exceptional
  divisors)~$=3h-\sum _{i=1}^{22} e_i$ (where we use the lower case letters to denote
  the homology classes represented by the exceptional divisors, and not by
  their potential strict transforms in further blow-ups) on
  $W=\cpk \# 22\cpkk$. Restrict this class to the complement of the
  neighbourhoods of the two chains ${\mathcal {C}}_1$ and ${\mathcal
  {C}}_2$, and extend it to $\widehat{W}$; denote the result (as
  well as its Poincar\'e dual) by $\widehat{K}$. Indeed, the class
  extends: (the Poincar\'e dual of) $K$ is the first Chern class of
  the standard symplectic structure on $W$, and as the rational
  blow-down of symplectic configurations is a symplectic operation
  \cite{PS, Sym}, not only the Chern class, but the symplectic form
  also extends; its first Chern class is (the Poincar\'e dual of)
  $\widehat{K}$.  Our goal is to show that
  $SW_{\widehat{W}}(\widehat{K})=\pm 1$.

  For this argument, we consider a metric $g$ on $W$ with the property
  that the neighbourhoods of ${\mathcal {C}}_1$ and ${\mathcal {C}}_2$
  are pulled away, that is, connected to their complement by a
  sufficiently long neck.  For such a metric, the gluing theorem of
  Seiberg-Witten solutions implies that
\[
  SW_{W, g}(K)=SW_{\widehat{W}}(\widehat{K}).
\]
  As noted above, the right-hand side is independent of the metric as
  we use small perturbation and we have $b_2^-(\widehat{W})=4<9$; the
  left-hand side, however, might depend on the choice of $g$, hence
  the appearance of the metric $g$ in the formula. To conclude the
  proof, we need to determine the Seiberg-Witten value in the chamber
  represented by the metric $g$ on $W$.

  The metric $g$ induces a period point (the unique self-dual 2-form
  with length 1 paring with the hyperplane class $h$ positively) which
  is in the chamber with homology classes orthogonal to the spheres in
  the two chains ${\mathcal {C}}_1$ and ${\mathcal {C}}_2$. As $K\cdot
  h =3>0$, the value of the Seiberg-Witten invariant in the chamber is
  decided by the sign of $K\cdot \alpha$ for a homology class $\alpha$
  which is
\begin{itemize}
\item of non-negative square,
\item orthogonal to the homologies of the spheres in
  ${\mathcal {C}}_1 \cup {\mathcal {C}}_2$,
\item $\alpha \cdot h>0$.
\end{itemize}
Indeed, as in the chamber containing $h$ the Seiberg-Witten invariant
vanishes, if $K\cdot \alpha<0$, we have (by the wall-crossing formula)
that $SW_{W,g}(K)=\pm 1$.
It is not hard to check that for the class
\begin{eqnarray*}
\alpha =
& 15h    -6e_1-5e_2 -3e_3-3e_4-4e_5-2e_6-6e_7-4e_8-4e_9-e_{10} \\
& -2e_{12}-e_{13}-e_{14}-2e_{15}-4e_{16}-3e_{18}-e_{19}-e_{20}-e_{21}-3e_{22}
\end{eqnarray*}
all the requirements above are satisfied, implying
$SW_{\widehat{W}}(\widehat{K})=\pm 1$. As $SW_{\cpk\#4\cpkk}\equiv 0$, this shows that the two four-manifolds are
not diffeomorphic, concluding the proof.} 
\end{proof}


{
\begin{rem} For completeness, the pairing $\alpha \cdot \alpha $ is equal
  to $10$ and $K\cdot \alpha$ is equal to $-12$, and obviously $\alpha
  \cdot h=15>0$. To check that $\alpha$ is, indeed,
  orthogonal to all the spheres in ${\mathcal {C}}_1$ and ${\mathcal {C}}_2$,
  below we list the homology classes of
  the spheres in ${\mathcal {C}}_1 \cup {\mathcal {C}}_2$ in
  the 'lower case' basis, the easy
  verification of orthogonality is left to the reader.
\end{rem}

The homology classes of the spheres in the blow-ups (in the basis
given by the exceptional divisors and not by their possible strict transforms,
that is, via our previous convention in lower case letters)
are given as follows:
\[
H = h-e_1-e_{18}-e_{19}-e_{20}-e_{21}-e_{22} \qquad L = h-e_7-e_8-e_9-e_{19}
\]
\[
F_1 = 3h-e_1-e_2-e_3-e_4-e_5-e_6-e_7-e_8-e_9-e_{13}-e_{14}-2e_{22}
\]
\[
E_1 = e_1-e_2-e_{10} \qquad E_3 = e_3-e_4-e_{11} \qquad E_4 = e_4-e_{12}-e_{13}
\]
\[
E_5 = e_5-e_6-e_{15} \qquad
E_6 = e_6-e_{15} \qquad
E_9 = e_9 - e_{16} - e_{17}
\]
\[
E_{13} = e_{13} - e_{14} \qquad
E_{19} = e_{19}-e_{20} \qquad
E_{20} = e_{20}-e_{21}
 \]
 \[
A = h-e_5-e_6-e_7-e_{18} \;
B = h-e_3-e_4-e_8-e_{12}-e_{18} \;
C = h-e_1-e_2-e_9-e_{17}
\]
\[
X = h-e_1-e_3-e_7 \qquad
Y = h-e_1-e_5-e_8-e_{10} \qquad
Z = h-e_3-e_5-e_9-e_{11}-e_{16}
\]

}

\section{ Proof of exoticness of $\widehat{W}$ via algebraic geometry}
\label{sec:AG}

We have Wahl chains $\mathcal{C}_1,\mathcal{C}_2$ in the complex projective surface $W=\cpk \# 22 \cpkk$. By Artin's contractibility criterion \cite[Theorem 2.3]{Art}, there is a normal projective surface $W'$ and a birational morphism $\sigma \colon W \to W'$ which contracts $\mathcal{C}_1$ and $\mathcal{C}_2$ into two singular points $P, Q \in W'$. This morphism induces an isomorphism between $W \setminus{\mathcal{C}_1 \cup \mathcal{C}_2}$ and $W' \setminus{P,Q}$. In this section, we give a proof of exoticness of $\widehat{W}$ via algebraic properties of $W'$. This is based on \cite[Section 1]{RU}.

Let us first recall some definitions from algebraic geometry to be
applied to our particular situation. A Wahl singularity is a
2-dimensional cyclic quotient singularity of type
$\frac{1}{p^2}(1,pq-1)$, where gcd$(p,q)=1$.
For a complex surface
$W'$ with only Wahl singularities, we consider its minimal resolution
$\sigma \colon W \to W'$, where $W$ is a nonsingular complex
projective surface and $\sigma$ is a birational morphism. The
pre-image of each Wahl singularity is a Wahl chain $\mathcal{C}_j=\cup_i C^j_i$, where
the $C_i^j$s are 2-spheres. The discrepancy $d(C_i^j)$ of $C_i^j$ is a
rational number such that 
$$K_W \equiv \sigma^*(K_{W'}) + \sum_j \sum_i d(C^j_i) \, C^j_i,$$
(see e.g. \cite[Definition 2.22]{KM}), where $K_W$, $K_{W'}$ are
canonical divisors, the symbol $\equiv$ is numerical equivalence of
divisors, and the sum runs through all the Wahl chains $\mathcal{C}_j$
over Wahl singularities with $C_i^j$ the $i^{{\rm {th}}}$ sphere in
the Wahl chain $\mathcal{C}_j$.  One can show that
$-1<d(C^j_i)<0$. To compute these numbers one uses the adjunction
formula and solves the linear system
$$-(C_k^j\cdot C_k^j)-2= \sum_i d(C_i^j) \, (C_i^j \cdot C_k^j), \ \ \text{for
  all} \ k,$$ for each Wahl chain. This system has a unique solution
since the intersection matrix is negative definite \cite{mum}. The
canonical divisor $K_{W'}$ is $\mathbb{Q}$-Cartier, and so there
is $u\in \NN$ such that $uK_{W'}$ is a Cartier divisor. This integer $u$ can
be taken as the least common multiple of the indices $p$ of the singularities
$\frac{1}{p^2}(1,pq-1)$ in $W'$. The canonical class $K_{W'}$ is \emph{ample}
if there is an integer $v\in \NN$ so that the sections of the line bundle
associated to $uvK_{W'}$ define an embedding of $W'$ into some
projective space. 

We now state the result we need from \cite{RU} to prove exoticness.

\begin{thm}\label{thm:algebraic}(Theorem 2.3 and Corollary 2.5 of \cite{RU})
Let $W'$ be a complex projective surface with only Wahl singularities such that $K_{W'}$ is ample. Then, the rational blow-down $\widehat{W}$ of the Wahl chains $\cup_j \mathcal{C}_j$ in the minimal resolution $W$ of $W'$ is a (smoothly) minimal symplectic 4-manifold. Moreover, if $\widehat{W}$ is simply-connected and it has an odd intersection form, then $\widehat{W}$ is an exotic
$$(2p_g(W)+1) \cpk \; \# \; (10p_g(W)+9-K_{W'}^2) \cpkk $$
with $10p_g(W)+9-K_{W'}^2>0$. Thus if $W$ is rational (e.g. it is some blow-up of $\cpk$), then $\widehat{W}$ is a minimal exotic $\cpk \# (9-K_{W'}^2) \cpkk$.
\end{thm}

Informally, singular surfaces in the
Koll\'ar--Shepherd-Barron--Alexeev compactification of the moduli
space of complex surfaces of general type (the analogue to the
Deligne-Mumford compactification of the moduli space of compact
Riemann surfaces of genus $g\geq2$) with only Wahl singularities
produce minimal exotic 4-manifolds.

The idea of the proof is to give a symplectic structure on
$\widehat{W}$ via the embedding of $W$ into a projective space defined
by a very ample line bundle of the form $\sigma^*(NK_{W'})-D$ where
$N\gg 0$ and $D$ is an effective sum of curves in the Wahl chains
contracted by $\sigma$. The symplectic form will be then in the class
of $N K_{W'}$. Thus the ampleness of $K_{W'}$ rules out the existence
of $(-1)$ 2-spheres in $\widehat{W}$. See the proof of \cite[Theorem
  2.3]{RU} for details.
\vspace{0.3cm}

We now return to our particular construction. Recall that we
considered the following pair of chains $\mathcal{C}_1,\mathcal{C}_2$
of equivariant curves in $W=\cpk \# 22 \cpkk$. We include the
intersection numbers and also the discrepancies to be used in the proof of
Theorem \ref{thm:nondiffeo}:
\[\begin{array}{*{19}c}
    \mathcal{C}_1 & \colon & B & - & C & - & A & - & E_6 & - & F_1 & - & E_9 & - & L & - & E_{19} & - & E_{20}\\
    \text{self-int}& \colon & -4 && -3 && -3 && -2 && -6 && -3 && -3 && -2 && -2\\
    \text{discr} & \colon & -\frac{47}{65} && -\frac{58}{65} && -\frac{62}{65} && -\frac{63}{65} && -\frac{64}{65} && -\frac{61}{65} && -\frac{54}{65} && -\frac{36}{65} && -\frac{18}{65},
\end{array}\]
and
\[\begin{array}{*{19}c}
    \mathcal{C}_2 & \colon & Y & - & E_5 & - & Z & - & H & - & E_1 & - & X & - & E_3 & - & E_4 & - & E_{13}\\
    \text{self-int}& \colon & -3 && -3 && -4 && -5 && -3 && -2 && -3 && -3 && -2\\
    \text{discr} & \colon & -\frac{49}{79} && -\frac{68}{79} && -\frac{76}{79} && -\frac{78}{79} && -\frac{77}{79} && -\frac{74}{79} && -\frac{71}{79} && -\frac{60}{79} && -\frac{30}{79}.
\end{array}\]

We will prove here the following:

\begin{thm}\label{thm:nondiffeo2}
The symplectic 4-manifold $\widehat{W}$ obtained by rationally blowing
down the two disjoint Wahl chains $\mathcal{C}_1$ and $\mathcal{C}_2$
is minimal.
\end{thm}


\begin{proof}
We have shown already in Section~\ref{sec:construction}  that $\widehat{W}$ is simply connected. Here we will apply Theorem \ref{thm:algebraic}, and so we only need to check that $K_{W'}$ is ample. To verify ampleness, by the
Nakai-Moishezon criterion (see e.g. \cite[Theorem 1.37]{KM}),
we need to prove that $K_{W'} \cdot \Gamma > 0$ for every
irreducible curve $\Gamma$ in $W'$. (We already know that
$K_{W'}^2=5>0$.) We remark that $W'$ is a projective surface with
two Wahl singularities $\frac{1}{65^2}(1,65\cdot18-1)$ for
$\mathcal{C}_1$ and $\frac{1}{79^2}(1,79 \cdot 30-1)$ for
$\mathcal{C}_2$. Let $\sigma \colon W \to W'$ be the contraction
as above. In what follows, we will use numerical
equivalence $\equiv$ for divisors, as we only need to check intersection numbers.

Let $S \to \mathbb{C} \mathbb{P}^1$ be the elliptic surface found in Section~\ref{sec:construction} to construct $W$. First, we know that $-K_S$ is linearly equivalent to a fiber, and so we have $$K_S \equiv -\frac{1}{3}(X + E_1 + Y + E_5 + Z + E_3) -
\frac{1}{3}(A + B + C) - \frac{1}{3}F_1.$$ Using the pull-back formula
for the canonical class after a blow-up, we obtain (using the previous
equation) that \begin{align*} K_W & \equiv -\frac{1}{3}(X + E_1 + Y +
  E_5 + Z + E_3 + A + B + C + F_1) \\ &\quad + \frac{2}{3}(E_{16} +
  E_{15} + E_{17} + E_{12})+ \frac{1}{3}(E_{11} + E_{10} + E_{18} +
  E_{22}) \\ &\quad + E_{19} + 2 E_{20} + 3 E_{21} + \frac{2}{3}
  E_{13} + \frac{4}{3} E_{14}.
\end{align*}
On the other hand, by the definition of discrepancy of exceptional divisors, we have the equality \begin{align*}
    \sigma^*K_{W'} &\equiv K_W + \frac{47}{65}B + \frac{58}{65}C +
    \frac{62}{65}A + \frac{63}{65}E_6 + \frac{64}{65}F_1 +
    \frac{61}{65}E_9 + \frac{54}{65}L \\ &\quad + \frac{36}{65}E_{19}
    + \frac{18}{65} E_{20} + \frac{49}{79}Y + \frac{68}{79}E_5 +
    \frac{76}{79}Z + \frac{78}{79}H + \frac{77}{79}E_1 \\ &\quad +
    \frac{74}{79}X + \frac{71}{79}E_3 + \frac{60}{79}E_4 +
    \frac{30}{79}E_{13}.
\end{align*}
Therefore, by plugging in $K_W$ in the last formula, we obtain that
$\sigma^*K_{W'}$ can be written as a sum of curves with positive rational coefficients (i.e. it is a $\mathbb Q$-effective divisor). We fix this representative of $\sigma^*K_{W'}$, and we refer to these curves as the support of $\sigma^*K_{W'}$.
To check that
$K_{W'}$ is ample, suppose $\Gamma \subseteq W'$ is a curve such that
$K_{W'}\cdot\Gamma \leq 0$. Let $\tilde\Gamma$ be its strict transform
in $W$. Then
\[K_{W'}\cdot\Gamma = \sigma^*K_{W'}\cdot\sigma^*\Gamma = \sigma^*K_{W'} \cdot \tilde\Gamma.\] Since $\sigma^*K_{W'}$ can be written as a $\mathbb Q$-effective divisor, any curve not included in its support must intersect $\sigma^*K_{W'}$ non-negatively. Moreover, $\tilde
\Gamma$ is a strict transform of a curve in $W'$, and so it is not contracted by $\sigma$. Thus  $\tilde
\Gamma$ is either one of the non-contractible curves in the support of $\sigma^*K_{W'}$, or it does not
intersect this support at all.

In the first case, $\tilde\Gamma$ can be one of the following curves:
$E_{11}$, $E_{16}$, $E_{15}$, $E_{10}$, $E_{17}$, $E_{12}$, $E_{18}$,
$E_{22}$, $E_{21}$ or $E_{14}$. Since all of them are $(-1)$-curves in
$W$, their intersection with $K_W$ is $-1$, and so we obtain
\begin{align*}
    \sigma^*K_{W'} \cdot E_{11} &= -1 + \frac{76}{79} + \frac{71}{79}
    > 0, & \sigma^*K_{W'} \cdot E_{16} &= -1 + \frac{76}{79} +
    \frac{61}{65} > 0,\\ \sigma^*K_{W'} \cdot E_{15} &= -1 +
    \frac{68}{79} + \frac{63}{65} > 0, & \sigma^*K_{W'} \cdot E_{10}
    &= -1 + \frac{49}{79} + \frac{77}{79} > 0,\\ \sigma^*K_{W'} \cdot
    E_{17} &= -1 + \frac{58}{65} + \frac{61}{65} > 0, & \sigma^*K_{W'}
    \cdot E_{12} &= -1 + \frac{47}{65} + \frac{60}{79} >
    0,\\ \sigma^*K_{W'} \cdot E_{18} &= -1 + \frac{47}{65} +
    \frac{62}{65} > 0, & \sigma^*K_{W'} \cdot E_{22} &= -1 +
    \frac{64}{65} + \frac{64}{65} > 0,\\ \sigma^*K_{W'} \cdot E_{21}
    &= -1 + \frac{18}{65} + \frac{78}{79} > 0, & \sigma^*K_{W'} \cdot
    E_{14} &= -1 + \frac{30}{79} + \frac{64}{65} > 0.\\
\end{align*}

This leaves with the case that $\tilde\Gamma$ does not intersect the
support of $\sigma^*K_{W'}$. But this is impossible as this support
contains $X$, $E_1$, $E_{10}$, $Y$, $E_5$, $E_{15}$, $Z$, $E_{16}$,
$E_{11}$, $E_3$, and these curves form a fiber of the induced elliptic
fibration $W \to \mathbb C \mathbb P^1$. But then $\tilde\Gamma$ must
be a component of some other singular fiber, and the only possibility
would be $Q$ (see Figure \ref{22up}), whose intersection with
$\sigma^*K_{W'}$ is not zero. Therefore such $\tilde\Gamma$ does not
exist, and so $K_{W'}$ is ample.
\end{proof}

\begin{rem}
As $K_{W'}$ is ample, the surface $W'$ lives in the
Koll\'ar--Shepherd-Barron--Alexeev moduli space of surfaces with
$K^2=5$ and $p_g=q=0$. The local-to-global obstruction for complex
deformations of $W'$ is in $H^2(W',T_{W'})$, and by \cite[\S 4]{RU}
one can compute that $H^2(W',T_{W'})=\mathbb C^2$ (almost the same
proof as for \cite[Proposition 5.4]{RU}). Therefore it is not trivial
to decide whether there is a complex ($\mathbb Q$-Gorenstein)
smoothing of $W'$ or not. If so, then this smoothing would produce
simply connected surfaces of general type with $K^2=5$ and
$p_g=0$. There are no such examples in the literature, this is an open problem. We also have
$H^1(W',T_{W'})=0$ by \cite[Theorem 4.1]{RU}, and so $W'$ is
equisingularly rigid.
\end{rem}

\end{document}